\documentclass[11pt]{amsart}

\usepackage{url,graphicx,tabularx,array,geometry}
\usepackage{graphicx}
\usepackage{amssymb}
\usepackage{fancyhdr}
\usepackage{wrapfig}
\usepackage{epstopdf}
\usepackage{amsmath}
\usepackage[labelfont=bf]{caption}
\usepackage[hidelinks]{hyperref}
\usepackage{amsthm}
\usepackage{bbm}
\usepackage[T1]{fontenc}
\usepackage{yfonts}
\usepackage{harpoon}
\usepackage[all]{xy}
\usepackage{tikz}
\usepackage{hyperref}
\usepackage{comment}
\hypersetup{
 colorlinks,
 linkcolor={red!50!black},
 citecolor={blue!50!black},
 urlcolor={blue!80!black}
}
\usepackage[maxbibnames = 50]{biblatex}
\renewbibmacro{in:}{}
\DeclareFieldFormat{pages}{#1}
\usetikzlibrary{decorations.markings}
\tikzstyle{vertex}=[circle, draw, inner sep=0pt, minimum size=6xpt]

\DeclareMathOperator{\diam}{diam}

\makeatletter
\def\imod#1{\allowbreak\mkern10mu({\operator@font mod}\,\,#1)}
\makeatother

\theoremstyle{plain}
\newtheorem{theorem}{Theorem}

\newtheorem{theoremN}{Theorem}[section]
\newtheorem{exampleN}[theoremN]{Example}
\newtheorem{propositionN}[theoremN]{Proposition}
\newtheorem{corollaryN}[theoremN]{Corollary}
\newtheorem{lemmaN}[theoremN]{Lemma}
\newtheorem{conjectureN}[theoremN]{Conjecture}

\theoremstyle{definition}
\newtheorem*{definition}{Definition}
\theoremstyle{remark}

\newtheorem*{remark}{Remark}
\newtheorem*{note}{Note}

\newcommand{\tightoverset}[2]{%
  \mathop{#2}\limits^{\vbox to -.5ex{\kern-0.75ex\hbox{$#1$}\vss}}}

\begin{document}
\title[Minimal prime graphs]{Minimal prime graphs of finite solvable groups}

\author[Florez, Higgins, Huang, Keller, and Shen]{Chris Florez, Jonathan Higgins, Kyle Huang, Thomas Michael Keller, Dawei Shen}

\address{Chris Florez, Department of Mathematics, David Rittenhouse Lab, University of Pennsylvania, 209 South 33rd Street, Philadelphia, PA 19104-6395, USA}
\email{cflorez@sas.upenn.edu}
\address{Jonathan Higgins, Mathematics and Computer Science Department, Wheaton College, 501 College Ave, Wheaton, IL 60187, USA}
\email{jonathan.higgins@my.wheaton.edu}
\address{Kyle Huang, Mathematics Department, University of California-Berkeley, 2227 Piedmont Avenue, Berkeley, CA 94720, USA}
\email{kyle.huang@berkeley.edu}
\address{Thomas M. Keller, Department of Mathematics, Texas State University, 601 University Drive, San Marcos, TX 78666-4616, USA}
\email{keller@txstate.edu}
\address{Dawei Shen, Department of Mathematics and Statistics, Washington University in St. Louis, 1 Brookings Dr., St. Louis, MO 63105, USA}
\email{shen.dawei@wustl.edu}


\begin{abstract}
We explore graph theoretical properties of minimal prime graphs of finite solvable groups.
In finite group theory studying the prime graph of a group has been an important
topic for the past almost half century. Recently prime graphs of solvable groups have
been characterized in graph theoretical terms only. This now allows the study of these
graphs with methods from graph theory only. Minimal prime graphs turn out to be of particular interest, and in this paper we pursue this further by exploring, among other things, diameters, Hamiltonian cycles and the property of being self-complementary for minimal prime graphs. We also study a new, but closely related notion of minimality for
prime graphs and look into counting minimal prime graphs.

\end{abstract}

\maketitle

\section{Introduction}

The topic of this paper are prime graphs of finite solvable
groups. The prime graph of a finite group is the graph whose vertices are the prime numbers
dividing the order of the group, with two vertices being linked by an edge if and only if their 
product divides the order of some element of the group. Prime graphs were introduced by
Gruenberg and Kegel in the 1970s and have been an object of continuous study since then.
They were among the first graphs assigned to groups. This idea of representing group theoretical
data via graphs and describing them via graph theoretical notions proved so successful 
that today there is a myriad of graphs (e.g. character degree graphs, conjugacy class size
graphs, etc.) and a whole industry of exploring them. For this reason, today prime graphs are
often referred to as Gruenberg-Kegel graphs.\\

While a focus in the study of prime graphs has been on simple groups for a long time,
the main result of \cite{2015_REU_Paper}, somewhat surprisingly, is a purely graph theoretical
characterization of prime graphs of solvable groups: A (simple) graph is the prime graph
of a finite solvable group if and only if its complement is triangle-free and 3-colorable.
This made it possible to study
simple groups whose prime graph is that of a finite solvable group, see 
\cite{Maslova_almost_simple}. Moreover, this characterization 
allowed the authors of \cite{2015_REU_Paper}, for solvable groups, to introduce and study 
the idea of
minimal prime graphs: connected graphs whose complement is triangle-free and 3-colorable,
but removing an edge means that the complement has a triangle or is no longer 3-colorable.
As outlined in detail in \cite{2015_REU_Paper}, the groups whose prime graphs are minimal are groups which are "saturated in Frobenius actions"
and have a restricted, but highly non-trivial structure. \\

In this paper we will turn our focus to graph theoretical properties of minimal prime graphs. Some important properties of minimal prime graphs were already proved in \cite{2015_REU_Paper}, such as having complement with chromatic number 3 and having an induced 5-cycle. We add to this by exploring diameters, Hamiltonian cycles, the property of being self-complementary, and others. We also take a look at counting minimal prime graphs. Moreover, we introduce an alternative notion of minimal prime graphs - minimally connected prime graphs -
which turns out to be closely related to minimal prime graphs.\\





\section{Minimally Connected Prime Graphs and Minimal Prime Graphs}

Below, we provide a formal definition for the Minimally Connected Prime Graphs that we will be studying throughout this paper.

\begin{definition} 
    A \textbf{minimally connected prime graph (MCPG)} is a graph $\Gamma$ with at least two vertices such that 
    \begin{enumerate}
        \item $\Gamma$ is connected
        \item $\overline{\Gamma}$ is triangle-free
        \item $\overline{\Gamma}$ is 3-colorable
        \item The removing of any edge results in a violation of at least one of the conditions (i), (ii), or (iii).
    \end{enumerate}\end{definition}
    
    Of particular interest are what we refer to as complete bridge graphs.

\begin{definition}
\label{complete bridge}
    A \textbf{complete bridge graph}, denoted $B_{m,n}$, is a graph on $m+n$ points, where $m$ of the points are totally connected with one another, and likewise for the other $n$ points, and the two subgraphs are connected by one edge, which we call the bridge. We take as convention that $m \geq n$. 
\end{definition}

\begin{lemmaN} \label{2-complete graphs are MCPG}
    Let $\Gamma = B_{m, n}$. Then $\Gamma$ is a MCPG if and only if $m, n > 1$ or $m \in\{1,2\}, n = 1$. 
\end{lemmaN}

\begin{proof} 
    Let $\Gamma =B_{m,n}$, where $m\geq n \geq 1$. It clearly has at least 2 vertices. The graph $\Gamma$ is also connected because the two complete subgraphs are linked by an edge. 
    
    Next, we will show that $\overline{\Gamma}$ is 2-colorable. In fact something stronger is true - obviously $\overline{\Gamma}$ is the complete bipartite graph with one edge removed. 
The fact that $\overline{\Gamma}$ is 2-colorable implies it is triangle-free. 
    

    Finally, we want to show that if we remove any edge from $\Gamma$, (at least) one of the first three conditions in the definition of minimally connected prime graphs is violated. Removing the bridge edge results in a disconnected graph.\\
    Let $n > 1$. Then removing an edge $\alpha \beta$ which is not the bridge edge, combined with a vertex $\gamma$ which is not connected to either $\alpha, \beta$ induces a triangle in the complement. Such a $\gamma$ exists since $m\geq n >1$. 
    
    Next let $m \in\{ 1,2\}, n = 1$. Then the removal of any edge will result in a disconnected graph.  
    
    Now let $\Gamma = B_{m,n}$ where $m >2, n=1$. Then $\Gamma$ contains $B_{m-1, 2}$ as a proper subgraph, which corresponds to moving the bridge vertex over to the other cluster. Thus the removal of an edge towards this subgraph demonstrates that $\Gamma$ is not minimal. 
\end{proof}
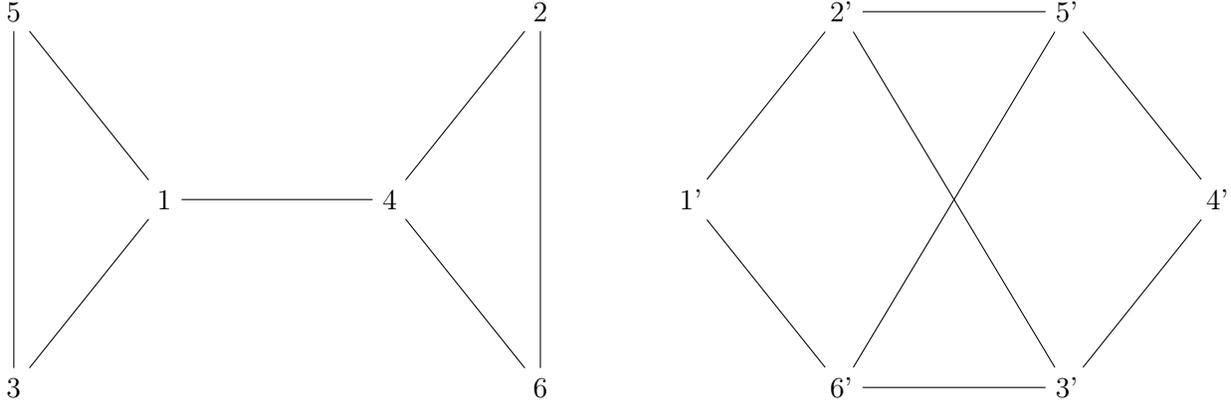
\begin{figure}
\label{fig1}
    \centering
\begin{tikzpicture}

  [scale=.7,auto=left,every node/.style={circle,fill=black!20}]
  \node (n6) at (1,10) {5};
  \node (n4) at (3,7.5)  {1};
  \node (n5) at (6,7.5)  {4};
  \node (n1) at (8,10) {2};
  \node (n2) at (8,5)  {6};
  \node (n3) at (1,5)  {3};
  \node (n7) at (12,10) {2'};
  \node (n8) at (15,10) {5'};
  \node (n9) at (10,7.5) {1'};
  \node (n10) at (12,5) {6'};
  \node (n11) at (15,5) {3'};
  \node (n12) at (17,7.5) {4'};

  \foreach \from/\to in {n6/n4,n4/n5,n5/n1,n1/n2,n2/n5,n6/n3,n3/n4,n9/n7,n7/n8,n8/n12,n12/n11,n11/n10,n10/n9,n7/n11,n8/n10}
    \draw (\from) -- (\to);
\end{tikzpicture}
\caption{The graph on the left is a MCPG where $|V(\Gamma[\pi])| = |V(\Gamma[\sigma])=3$. The graph on the right is its complement, where the primed points correspond to the numbered points on the left. This demonstrates the point above---if we let the odd points be one color (blue) and the evens be another (green), then the complement graph is seen to be a  2-colorable hexagon, and all of the blue points are connected with all the green points, excluding one edge.}
\end{figure}

Next, we will define some more notation. Let the set of MCPGs defined in Lemma \ref{2-complete graphs are MCPG} be denoted as $\mathcal{C}$.

From Theorem \ref{2-complete graphs are MCPG}, it is natural to wonder whether a natural extension of bridge 
having three disjoint, complete induced subgraphs can also be a MCPG, but it is immediate to see that then one can
pick vertices from each of the three complete subgraphs which do not have any edge between them, so the complement
has a triangle.\\


%

Next we recall the notion of a minimal prime graph as it was defined in \cite{2015_REU_Paper}.

\begin{definition} 
    A \textbf{minimal prime graph} is a graph $\Gamma$ with at least two vertices such that 
    \begin{enumerate}
        \item $\Gamma$ is connected
        \item $\overline{\Gamma}$ is triangle-free
        \item $\overline{\Gamma}$ is 3-colorable
        \item The removing of any edge results in a violation of at least one of the conditions (ii) or (iii).
    \end{enumerate}
\end{definition}

In particular, one can potentially remove an edge from a minimal prime graph to attain a disconnected prime graph. 

We let $\mathcal{G}$ denote the set of minimal prime graphs and $\hat{\mathcal{G}}$ denote the set of minimally connected prime graphs. From the similarity of the definitions we immediately get the following Lemma:

\begin{lemmaN} \label{MPGs are MPCGs}
    $\mathcal{G} \subsetneq \hat{\mathcal{G}}$
\end{lemmaN}

\begin{proof} Let $\Gamma$ be a minimal prime graph. Then removing any edge violates either condition (ii) or (iii), so in particular we can see that removing any edge violates either condition (i), (ii), or (iii). To see that the inclusion is strict, consider the example in Figure 1. Removing the middle edge still yields a prime graph but is disconnected. A simpler example is the connected graph on 2 vertices.  

\end{proof}

From here we can classify the graphs which are minimally connected prime graphs but \textbf{not} minimal prime graphs. We will see that the graphs in $\hat{\mathcal{G}} \setminus \mathcal{G}$ exactly take the form of complete bridge graphs described in Theorem \ref{2-complete graphs are MCPG}. We see that $\mathcal{C}$ is the set of all complete bridge graphs, with some small exceptions. 

\begin{theoremN} \label{2-complete graphs are the set difference}
    Let $\Gamma \in \hat{\mathcal{G}} \setminus \mathcal{G}$. Then $\Gamma = B_{m, n}$ for some $m, n$ where $m \geq n > 1$ or $m \in \{1, 2\}, n = 1$.
\end{theoremN}

\begin{proof} 
    By $\Gamma \in \hat{\mathcal{G}} \setminus \mathcal{G}$, we note that there necessarily exists an edge in $\Gamma$ such that, if removed, results in a disconnected prime graph. Call the resulting connected components $\Gamma_1$ and $\Gamma_2$, and suppose that $\Gamma_1$ has $m$ vertices and $\Gamma_2$ has $n$ vertices, with $m \geq n \geq  1$. 
    
    Suppose that $\Gamma_1$ is not a complete graph. If not, one can pick two vertices $v_1, v_2$ not linked by an edge, and another vertex $u \in \Gamma_2$ which is not a bridge vertex, to form a triangle in the complement. A similar reasoning applies to $\Gamma_2$. This shows that $\Gamma = B_{m,n}$ and so we are done by Lemma
    \ref{2-complete graphs are MCPG}.
\end{proof}
One desirable property of MCPGs, and our primary motivation in consdering such graphs, is that they properly satisfy minimality. This is reflected in the following Lemma, which assists us greatly in later describing the diameter in Theorem \ref{classifying graphs by their diameter}. 

\begin{lemmaN} \label{Minimality of MCPGs}
    Let $\Gamma$ be a MCPG with $n$ vertices. Then for any subgraph $\Lambda$ which satisfies (i), (ii), or (iii) on $n$ vertices, $\Lambda = \Gamma$. 
\end{lemmaN}
One can think of this as minimality in the poset of connected prime graphs on $n$ vertices. 
\begin{proof}
    Suppose $\Lambda$ is a proper subgraph of $\Gamma$, i.e. it is missing an edge in $\Gamma$. By $\Gamma$ being a MCPG, we have that removing this edge results in a graph which (1) is disconnected, (2) has a triangle in the complement, or (3) whose complement is not 3-colorable. \\[6pt]
    It is clear that removing more edges from $\Gamma$ cannot undo any of (1), (2), or (3),
    which implies the assertion.
\end{proof}
\begin{corollaryN} \label{Edge removal}
    Every MCPG on $n$ vertices can be attained by starting with the complete graph on $n$ vertices and removing edges until the removal of any edge causes the graph to violate (i), (ii), or (iii). 
\end{corollaryN}
\begin{proof}
    This follows from the above minimality lemma and also from the fact that a graph containing a minimal prime graph as a subgraph is itself a connected prime graph. 
\end{proof}

\section{Generating Prime Graphs}

In this section, we will be considering what kind of graphs can be generated from other graphs in $\mathcal{C},\mathcal{G}$, and (more broadly) $\hat{\mathcal{G}}$. First, we define what it means for a graph to be generated from another. 

\begin{definition}
    We say a prime graph $\Gamma^*$ can be \textbf{generated} from another graph $\Gamma$ if we can construct $\Gamma^*$ from $\Gamma$ by adding one new vertex $\alpha$ to $\Gamma$ and a collection of edges $\alpha \beta_i$, where $\beta_i \in \Gamma$ for all $i$. In other words, $|V(\Gamma^*)| = |V(\Gamma)|+1$ and $\Gamma^*[V(\Gamma)] = \Gamma$, where $\Gamma^*[V(\Gamma)]$ is the induced subgraph of $\Gamma^*$ on the vertices of $\Gamma$.
\end{definition}

\begin{lemmaN} \label{MPG to MPG construction}
    If $\Gamma \in \mathcal{G}$, then $\Gamma$ can be used to generate another graph $\Gamma^* \in \mathcal{G}$.
\end{lemmaN}

\begin{proof}
    This is proved as Proposition 3.1 in \cite{2015_REU_Paper}.
\end{proof}

\begin{lemmaN} \label{non-MPG MCPG to non-MPG MCPG construction}
    If $\Gamma \in \mathcal{C}$, then $\Gamma$ can be used to generate another graph $\Gamma^* \in \mathcal{C}$.
\end{lemmaN} 

\begin{proof} 
    First note that by Theorem \ref{2-complete graphs are the set difference} that $\Gamma$ is a bridge graph $B_{m,n}$ with $m, n > 1$ or $m= 1,2, n=1$. Then take one of the complete induced subgraphs, either of size $m, n$ and increase its size by one. Except in the case where $m= 2, n=1$ and the subgraph chosen is the one of size $n$, this produces a valid new bridge graph by Lemma \ref{2-complete graphs are MCPG}.
\end{proof}

\begin{figure}
    \centering
\begin{tikzpicture}

  [scale=.7,auto=left,every node/.style={circle,fill=black!20}]
  \node (n6) at (3,10) {5};
  \node (n4) at (5,7.5)  {1};
  \node (n5) at (8,7.5)  {4};
  \node (n1) at (10,10) {2};
  \node (n2) at (10,5)  {6};
  \node (n3) at (3,5)  {3};
  \node (n7) at (1,7.5) {$\alpha$};

  \foreach \from/\to in {n6/n4,n4/n5,n5/n1,n1/n2,n2/n5,n6/n3,n3/n4,n7/n6,n7/n4,n7/n3}
    \draw (\from) -- (\to);
\end{tikzpicture}
\caption{This is an example of a graph that can be generated from the MCPG in Figure 1. Here, we introduce the new point $\alpha$, which we connect with all the odd-labeled vertices in the left induced subgraph to keep it complete.}
\end{figure}
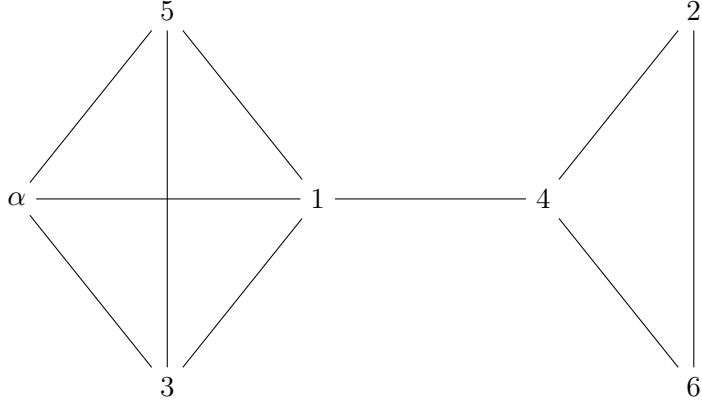

\begin{corollaryN} \label{MCPG to MCPG construction}
    If $\Gamma \in \hat{\mathcal{G}}$, then $\Gamma$ can be used to generate another graph $\Gamma^* \in \hat{\mathcal{G}}$. 
\end{corollaryN}

\begin{proof}
    Suppose $\Gamma$ is a MCPG. By Theorem \ref{2-complete graphs are the set difference}
    we know that $\Gamma$ is either a minimal prime graph or a complete bridge graph. If $\Gamma \in \mathcal{G}$, then a new minimal prime graph $\Gamma^*$ can be generated from $\Gamma$ by 
    \ref{MPG to MPG construction}, which is also a MCPG.
     If $\Gamma \in \mathcal{C}$, then a similar result follows from Lemma \ref{non-MPG MCPG to non-MPG MCPG construction}. 
\end{proof}

Corollary \ref{MCPG to MCPG construction} used the fact that graphs in $\mathcal{G}$ can be used to generate other graphs in $\mathcal{G}$ and similarly for graphs in $\mathcal{C}$, but this provides us with the question: Can a graph in $\mathcal{G}$ be used to generate a graph in $\mathcal{C}$? And how about in the reverse direction? We answer these questions in the following two lemmas.

\begin{lemmaN} \label{non-MPG MCPG to MPG construction}
    If $\Gamma \in \mathcal{C}$, where the two complete subgraphs of $\Gamma$ have at least two vertices each, then $\Gamma$ can be used to generate another graph $\Gamma^* \in \mathcal{G}$. 
\end{lemmaN}
\begin{proof} 
    Let $\Gamma \in \mathcal{C}$, as defined in the proposition. We introduce a new vertex $\alpha$ and we add every edge $\alpha \beta$, where $\beta \in V(\Gamma)$, excluding the bridge vertices $p$ and $q$. Let this new graph be $\Gamma^*$, which we will show is in $\hat{\mathcal{G}}\setminus\mathcal{C}$.
    \\
    \\
    First, we will show $\Gamma^* \in \hat{\mathcal{G}}$. Clearly, $\Gamma^*$ is connected because $\Gamma$ is connected and the vertex $\alpha$ has at least two edges by our specifications on $\Gamma$. Also, we know $E(\overline{\Gamma^*}) = E(\overline{\Gamma}) \cup \{\alpha p\} \cup \{\alpha q\}$. The complement graph $\overline{\Gamma}$ is triangle-free because $\Gamma \in \mathcal{C}$, and $\alpha, p$, and $q$ do not form a triangle because $pq \not\in E(\overline{\Gamma})$. Hence, we conclude $\overline{\Gamma^*}$ is triangle-free. Additionally, it is easy to show that $\overline{\Gamma^*}$ must be 3-colorable, which we know is a necessary property of minimal prime graphs by Lemma 3.2 in the prime graphs paper. By Theorem \ref{2-complete graphs are the set difference}, we know that $\Gamma$ is a 2-connected graph, which implies that $\overline{\Gamma}$ is 2-colorable. Because $\alpha$ is connected to vertices in both of the complete subgraphs, it must be a different color than the two used to color the points in the complement subgraphs. Hence, $\overline{\Gamma^*}$ is three-colorable. Now, consider the graph $\Gamma^* \setminus \{ab\}$. If $ab \in E(\Gamma)$ and $ab \neq pq$, then we know $\overline{\Gamma^*}$ is not triangle-free because this will induce a triangle in the subgraph $\overline{\Gamma}$. Next, suppose $ab=pq$. This results in the triangle with vertices $\alpha, p$, and $q$ in the complement graph. Finally, suppose $ab \in E(\Gamma^*)\setminus E(\Gamma)$. Hence, we can say $a=\alpha$. If $b$ is in the complete induced subgraph of $\Gamma$ containing $p$, then this creates the triangle of vertices $a,b,$ and $q$ in the complement graph. If $b$ is in the other complete induced subgraph, then we get the triangle of vertices $a,b,$ and $p$. Thus, removing any edge $ab$ from $\Gamma^*$ results in a triangle in $\overline{\Gamma^* \setminus \{ab\}}$, so we can conclude $\Gamma^* \in \hat{\mathcal{G}}$.
    \\
    \\
    Finally, we must show that $\Gamma^* \setminus \{ab\}$ does not result in a disconnected graph for any $ab \in E(\Gamma^*)$. If we consider the subgraph $\Gamma$, $pq$ is the only edge whose removal results in a disconnected graph. This means that $\Gamma^* \setminus \{ab\}$, where $ab \in E(\Gamma)$ and $ab \neq pq$, is still connected. Also, we find $\Gamma^* \setminus \{pq\}$ is still connected because $\alpha$ is connected to the two complete induced subgraphs of $\Gamma$ at a minimum of one vertex each. Finally, we consider $\Gamma^* \setminus\{ab\}$, where $ab \in E(\Gamma^*)\setminus E(\Gamma)$ (so, once again, we can let $a=\alpha$). This graph is still connected because the subgraph $\Gamma$ remains connected and $\alpha$ is connected to at least two other vertices. Hence, we see that removing any edge of $\Gamma^*$ does not result in a disconnected graph. However, for any graph in $\mathcal{C}$, there is an edge whose removal results in a disconnected graph, so we deduce $\Gamma^* \not\in \mathcal{C}$.
    \\
    \\
    It follows from the results above that $\Gamma^* \in \hat{\mathcal{G}}\setminus \mathcal{C} = \mathcal{G}$ (by Theorem \ref{2-complete graphs are the set difference}) and $\Gamma^*$ has one more vertex than $\Gamma$.
\end{proof}

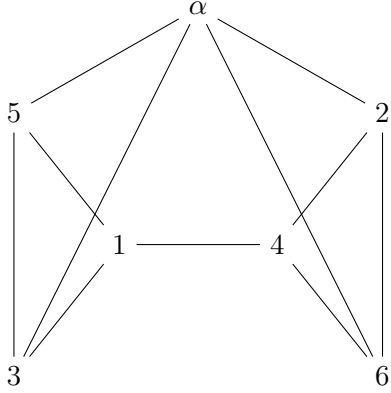
\begin{figure}
    \centering
\begin{tikzpicture}
  [scale=.7,auto=left]
  \node (n6) at (3,10) {5};
  \node (n4) at (5,7.5)  {1};
  \node (n5) at (8,7.5)  {4};
  \node (n1) at (10,10) {2};
  \node (n2) at (10,5)  {6};
  \node (n3) at (3,5)  {3};
  \node (n7) at (6.5,12) {$\alpha$};

  \foreach \from/\to in {n6/n4,n4/n5,n5/n1,n1/n2,n2/n5,n6/n3,n3/n4,n7/n6,n7/n1,n7/n2,n7/n3}
    \draw (\from) -- (\to);
\end{tikzpicture}
\caption{This is a graph in $\mathcal{G}$ that can be generated from the graph in Figure 1 through the method used in the proof of Lemma \ref{non-MPG MCPG to MPG construction}.}
\end{figure}

By the formulation of the lemma, the smallest MCPG that we can apply this construction method to is $B_{2,2}$. Whenever we do so, this produces the 5-cycle, which is the smallest MPG. In the following lemma, we show that there is no analogous construction in the reverse direction.

\begin{lemmaN} \label{No construction from MPG to non-MPG MCPG}
If $\Gamma \in \mathcal{G}$, then $\Gamma$ cannot be used to generate some graph $\Gamma^* \in \mathcal{C}$. 
\end{lemmaN}

\begin{proof}
Recall that by Lemma 3.2 in \cite{2015_REU_Paper} we know that the complement of any graph in $\mathcal{G}$ has chromatic number 3, whereas complements of graphs in $\mathcal{C}$ are 2-colorable. This immediately implies the assertion.
\end{proof}

To summarize some our results on generation of graphs, we let $A \rightarrow B$ denote that "a prime graph in set $A$ can be used to generate a prime graph in set $B$." Clearly, $A \not\rightarrow B$ denotes "a prime graph in set $A$ cannot be used to generate a prime graph in set $B$." This leads us to the main theorem of this section.

\begin{theoremN} \label{Collecting Construction Lemmas}
    The following list shows which types of primes graphs can be used to generate other kinds of prime graphs:
    \begin{enumerate}
        \item  $\mathcal{G} \rightarrow \mathcal{G}$
        \item  $\mathcal{C} \rightarrow \mathcal{C}$
        \item  $\mathcal{C} \rightarrow \mathcal{G}$
        \item  $\mathcal{G} \nrightarrow \mathcal{C}$.
    \end{enumerate}
\end{theoremN}

\begin{proof}
The theorem follows from Lemmas \ref{MPG to MPG construction}, \ref{non-MPG MCPG to non-MPG MCPG construction}, \ref{non-MPG MCPG to MPG construction}, and \ref{No construction from MPG to non-MPG MCPG}.
\end{proof}

These are all of the interesting combinations of the graph classes. Also, we remind the reader that Corollary \ref{MCPG to MCPG construction} shows that any MCPG can be used to generate another MCPG, and we can deduce by Theorem \ref{2-complete graphs are the set difference} and Theorem \ref{Collecting Construction Lemmas} that any MCPG can be used to generate a graph in $\mathcal{G}$, but it is not the case that any MCPG can be used to generate a new graph in $\mathcal{C}$.

\section{The Diameters of Minimally Connected Prime Graphs}

Recall that the diameter of a graph is the maximum distance between a pair of vertices in the graph, where distance is the length of the shortest path. We take as convention that disconnected graphs have infinite diameter. We first consider the diameter of graphs in $\mathcal{C}$.

\begin{propositionN} \label{MCPG diameter}
    If $\Gamma \in \mathcal{C}$, then $\diam(\Gamma) \leq 3$. More precisely, $\diam(B_{1,1}) =1, \diam(B_{2,1})=2$, and $\diam(B_{m,n}) =3$ for $m \geq n >1$.
\end{propositionN}

\begin{proof} This is immediately checked by the reader.
    
    

    

\end{proof}

We can conclude our study of the diameters of MCPGs by using an important result from Lucido. 

\begin{lemmaN} \label{MCPG diameter upper-bound}
    If $\Gamma \in \hat{\mathcal{G}}$, then $\diam(\Gamma)\leq 3$.
\end{lemmaN}

\begin{proof}
    Because $\Gamma \in \hat{\mathcal{G}}$, we know that $\overline{\Gamma}$ is triangle-free and 3-colorable. By Theorem 2 in \cite{2015_REU_Paper}, we know that there is some solvable group $G$ such that $\Gamma$ is isomorphic to the prime graph of $G$. Lucido showed in \cite{Lucido1999} 
    that the diameter of the prime graph of a solvable group is less than or equal to 3, which establishes the lemma.
\end{proof}

\begin{theoremN} \label{diameter 3 implies 2-colorable}
    For any prime graph $\Gamma$ with $\diam(\Gamma) = 3$, $\overline{\Gamma}$ is 2-colorable.
\end{theoremN}

\begin{proof}
    By $\diam{\Gamma} = 3$, there necessarily exist points $a$, $b$ such that the distance between $a$ and $b$ is 3. 
    
    Since $a, b$ are not connected in $\Gamma$ they are necessarily connected in the complement, and thus also necessarily different colors. Let $a$ be colored green, $b$ be colored blue. 
    
    Let $v$ be any another vertex. Then if $v$ is connected to neither $a$ or $b$ in $\bar{\Gamma}$, there exists a path of length $2$ from $a$ to $b$ in $\Gamma$, $a \rightarrow v \rightarrow b$. 
    
    Suppose $v$ is connected to $a$ in $\bar{\Gamma}$. Then we can color $v$ blue. Otherwise if $v$ is connected to $b$ in $\bar{\Gamma}$, we color $v$ green. Note that $v$ cannot be connected to both, as otherwise there is a triangle $(a, b, v)$.
    
    We follow this procedure for all vertices which are not $a$ or $b$, and claim this is a valid $2$-coloring of $\bar{\Gamma}$. To see why, suppose we have two points $v_1, v_2$ which are both green under this construction. Then $v_1, v_2$ are connected to $b$ by construction. So $v_1, v_2$ cannot be connected, otherwise $\bar{\Gamma}$ contains a triangle $(v_1, v_2, b)$. A similar reasoning holds if $v_1, v_2$ are both blue. 
\end{proof}
\begin{corollaryN} \label{MPGs have diameter 2}
    For $\Gamma \in \mathcal{G}$, $\diam{\Gamma} = 2$.
\end{corollaryN}
\begin{proof}
    This follows from Theorem \ref{diameter 3 implies 2-colorable} and the fact that MPGs have chromatic number 3, from \cite{2015_REU_Paper}. 
\end{proof}

\begin{lemmaN}\label{edge removal doesn't decrease diameter}
    Let $\Gamma$ be any undirected graph, and let $\Gamma'$ be attained from $\Gamma$ via removing an edge. Then $\diam(\Gamma') \geq \diam(\Gamma)$  
\end{lemmaN}
\begin{proof}
    This is immediate from considering the definition of diameter - removing edges cannot decrease the diameter, only increase it. 
\end{proof}
\begin{lemmaN} \label{Characterization of prime graphs w/ diameter 3}
    Let $\Gamma$ be a connected prime graph. Then $\diam(\Gamma) = 3$ if and only if any sequence of edge removals on $\Gamma$ to reach a MCPG ends in a complete bridge graph. 
\end{lemmaN}
\begin{proof}
    First recall that we can remove edges until reaching a MCPG by MCPGs satisfying the proper minimality conditions in Lemma \ref{Minimality of MCPGs} and Corollary \ref{Edge removal}. 
    
    Let $\Gamma$ be a prime graph such that there exists a sequence of edge removals to reach a minimal prime graph, which has diameter 2 by \ref{MPGs have diameter 2}. Then it follows from Lemma \ref{edge removal doesn't decrease diameter} that $\diam(\Gamma) \leq 2$, so in particular $\Gamma$ cannot be a minimal prime graph. The theorem follows from our discussion on the set-difference between MCPGs and MPGs. 
\end{proof}
    
    The following theorem characterizes prime graphs of solvable groups via our results so far. 
    
\begin{theoremN} \label{classifying graphs by their diameter}
    Let $\Gamma$ be a prime graph. Then
    \begin{itemize}
        \item[(i)] $\Gamma$ is disconnected if and only if the disconnected components are complete graphs
        \item[(ii)] $\diam(\Gamma) = 1$ if and only if $\Gamma$ is a complete graph
        \item[(iii)] $\diam(\Gamma) = 3$ if and only if every sequence of edge removals towards a minimally connected prime graph results in a complete bridge graph
        \item[(iv)] Otherwise, $\diam(\Gamma) = 2$
    \end{itemize}
\end{theoremN}
\begin{proof}
    $(i)$ follows from $\Gamma$ being triangle-free, so a disconnected graph must have complete components. 
    
    $(ii)$ follows easily from the definition of diameter.
    
    $(iii)$ is from \ref{Characterization of prime graphs w/ diameter 3}
\end{proof}

\section{Hamiltonian Cycles and Paths in MCPGs}
\begin{propositionN}
\label{2-complete graphs have ham path}
Any complete bridge graph contains a Hamiltonian path
\end{propositionN}
\begin{proof}
It is obvious that a complete bridge graph $B_{m,n}$ is not Hamiltonian. Even though it contains two complete induced subgraphs, the single edge connecting them renders it impossible for there to be a single cycle that includes every vertex. However, we can find a path that connects all vertices. Let $\{u_1, ..., u_m\}$ be the induced complete subgraph on $m$ vertices, and $\{v_1, ... v_n\}$ be the complete induced subgraph on $n$ vertices, where $u_m, v_n$ are the bridge vertices. Then we can construct the following path that reaches each vertex once: $u_1 \rightarrow u_2 \rightarrow ... \rightarrow u_{m-1} \rightarrow u_m \rightarrow v_n \rightarrow v_{n-1} \rightarrow ... \rightarrow v_2 \rightarrow v_1$. This is one such Hamiltonian path. 
\end{proof}
It is well-known that the number of Hamiltonian paths for a graph of the form $B_{n,n}$ (also known as a barbell graph) is $[(n-1)!]^2$ \cite{website:Wolfram-Alpha}. We generalize by considering the number of Hamiltonian paths for an arbitrary complete bridge graph. 

\begin{propositionN}
\label{number of ham paths}
The number of Hamiltonian paths on $B_{m,n}$ is $(m-1)!(n-1)!$.
\end{propositionN}

\begin{proof}
Consider the graph $\Gamma = B_{m,n}$. Any Hamiltonian path in $\Gamma$ must contain the full Hamiltonian paths of the complete subgraphs, connected by the bridge edge. The number of Hamiltonian paths for the complete graph $K_n$ is $n!$, so we attain $m!$ and $n!$ Hamiltonian paths from each induced complete subgraph. However, for a Hamiltonian path in $\Gamma$, we can say without loss of generality that the path in one complete subgraph must end at the bridge vertex and the other must start at the other bridge vertex. Thus, the number of valid Hamiltonian paths in the two subgraphs are $(m-1)!$ and $(n-1)!$, so when we consider all the Hamiltonian paths in $\Gamma$, we find that there are $(m-1)!(n-1)!$ in total.
\end{proof}

Now, we are prepared to show that all MPGs are Hamiltonian.

\begin{theoremN} \label{All MPGs are Hamiltonian}
    If $\Gamma \in \mathcal{G}$, then $\Gamma$ is Hamiltonian.
\end{theoremN}
\begin{proof}
    First recall from Lemma 4.1 of \cite{2015_REU_Paper} that all MPGs contain an induced 5-cycle. By some simple casework on the complement's coloring, we see that we can assume without loss of generality that 2 are green, 2 are blue, 1 is yellow, and they are labeled and connected as in Figure 4.
\begin{figure} \label{All MPGs are Hamiltonian Figure}
    \centering
    \begin{tikzpicture}
        [scale=.7,auto=left]
        \node[circle,fill=green!20] (n1) at (0, 2) {$g_1$};
        \node[circle,fill=green!20] (n2) at (0, 0) {$g_2$};
        \node[circle,fill=blue!20] (n3) at (3, 0) {$b_1$};
        \node[circle,fill=blue!20] (n4) at (3, 2) {$b_2$};
        \node[circle,fill=yellow!20] (n5) at (6, 4.5) {$y_1$};
        \foreach \from/\to in {n1/n2,n2/n3,n3/n4,n4/n5,n5/n1}
    \draw (\from) -- (\to);
    \end{tikzpicture}
    \caption{A 3-coloring of the complement of the induced 5-cycle in a minimal prime graph, superimposed on the 5-cycle}
\end{figure}
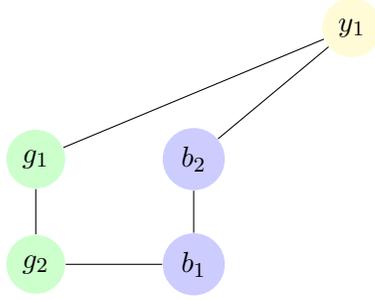
    
    Since green vertices cannot be connected by a edge in the complement, we see that they are fully connected amongst one another in $\Gamma$, and similarly for the blue and yellow vertices themselves. \\
    Let $\mathfrak{G}$ denote the set of green vertices, $\mathfrak{B}$ the blue vertices, and $\mathfrak{Y}$ the yellow vertices. Let the vertices be as labeled in the diagram. Suppose $\Gamma$ contains only that yellow vertex, $y_1$. Then we can create a Hamiltonian cycle as follows - start at $g_1$, go through every green vertex and end at $g_2$, then go to $b_2$, then go through every blue vertex ending on $b_1$, then go to $y_1$, and then back to $g_1$. This can be visualized as below:
    $$g_1 \rightarrow \text{ every other green vertex } \rightarrow g_2 \rightarrow b_1 \rightarrow \text{ every other blue vertex } \rightarrow b_2 \rightarrow y_1 \rightarrow g_1$$ 
    So now suppose there are other yellow vertices. If any yellow vertex $y$ is connected to $g_1$ then we have a Hamiltonian cycle as 
    $$g_1 \rightarrow \text{ every other green vertex } \rightarrow g_2 \rightarrow b_1 \rightarrow \text{ every other blue vertex } \rightarrow b_2 \rightarrow y_1 $$ $$\rightarrow \text{ every other yellow vertex } \rightarrow y \rightarrow g_1$$ 
    So now suppose no other yellow vertex $y$ is connected to $g_1$. By considering the triangles $(y, g_1, b_1)$, $(y, g_1, b_2)$ we see that all other yellow vertices $y$ are connected to $b_1$ and $b_2$. \\
    So then we can take the path
    $$g_1 \rightarrow \text{ every other green vertex } \rightarrow g_2 \rightarrow b_1 \rightarrow \text{ every other blue vertex } \rightarrow b_2 \rightarrow y$$ $$\rightarrow \text{ every other yellow vertex } \rightarrow y_1 \rightarrow g_1$$
    Thus $\Gamma$ is Hamiltonian. 
\end{proof}

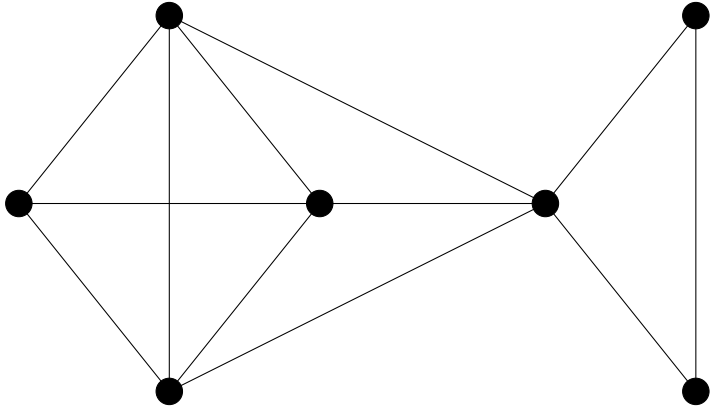
\begin{figure}\label{nonHamiltonian solvable group prime graph}
        \centering
        \begin{tikzpicture}
        
          [scale=.20,auto=left,every node/.style={circle,fill=black!20}]
          \node (n6) at (3,10) {};
          \node (n4) at (5,7.5)  {};
          \node (n5) at (8,7.5)  {};
          \node (n1) at (10,10) {};
          \node (n2) at (10,5)  {};
          \node (n3) at (3,5)  {};
          \node (n7) at (1,7.5) {};
          
        \draw [fill=black] (3,10) circle (5pt);
        \draw [fill=black] (5,7.5) circle (5pt);
        \draw [fill=black] (8,7.5) circle (5pt);
        \draw [fill=black] (10,10) circle (5pt);
        \draw [fill=black] (10,5) circle (5pt);
        \draw [fill=black] (3,5) circle (5pt);
        \draw [fill=black] (1,7.5) circle (5pt);
          \foreach \from/\to in {n6/n4,n4/n5,n5/n1,n1/n2,n2/n5,n6/n3,n3/n4,n7/n6,n7/n4,n7/n3, n6/n5, n3/n5}
            \draw (\from) -- (\to);
        \end{tikzpicture}
        \caption{An example of a solvable group prime graph which contains a complete bridge graph and is not Hamiltonian}
    \end{figure}

Like in Theorem \ref{classifying graphs by their diameter}, where minimality let us classify the prime graphs of solvable groups, we have a similar result for the Hamiltonian property. First note that a graph which is Hamiltonian after edge removals is Hamiltonian prior to the edge removals. So the only non-Hamiltonian prime graphs are those which contain a complete bridge graph (See Figure 5), and all extra edges share the same bridge vertex. If not, then a Hamiltonian cycle can be found. This leads to the following theorem:

\begin{theoremN}
    Let $\Gamma$ be a solvable group prime graph. Then $\Gamma$ is not Hamiltonian if and only if it is disconnected with 2 complete components, or contains a complete bridge graph and any extra edges share the same bridge vertex. 
\end{theoremN}

\section{Minimally Connected Prime Graphs That Are Self-Complementary}

\begin{lemmaN}
\label{2-complete sc}
If $\Gamma \in \mathcal{C}$ and $\Gamma$ is self-complementary, then $\Gamma = B_{2,2}$.
\end{lemmaN}

\begin{proof}
We know that $\mathcal{C}$ consists of the complete bridge graphs, $B_{1,1}$, $B_{2,1}$, and $B_{m,n}$, where $m \geq n>1$. From Observation 1 of \cite{sc-graphs}, we see that if $\Gamma$ is a self-complimentary graph, then $|V(\Gamma)| \equiv 0$ or $1 \pmod{4}$. This immediately rules out $B_{1,1}$ and $B_{2,1}$, so we will continue by only considering the graphs where $m, n >1$. 

If either $m, n \geq 3$ then we see that $B_{m,n}$ has $K_3$ as a subgraph, and thus has chromatic number at least $3$. However, we saw in \cite{Chris_counterexample} that the complement of $B_{m,n}$ is 2-colorable, thus $B_{m, n}$ cannot be self-complementary when $m \geq 3$ or $n \geq 3$. One can easily check by hand that $B_{2, 2}$ is self-complementary. 
\end{proof}

\begin{lemmaN}
\label{MPG sc}
If $\Gamma \in \mathcal{G}$ and $\Gamma$ is self-complementary, then $\Gamma = C_5$ (the 5-cycle).
\end{lemmaN}

\begin{proof}
Suppose $\Gamma \in \mathcal{G}$ and $\Gamma$ is self-complementary. We know $\overline{\Gamma}$ is triangle-free, so $\Gamma$ must also be triangle-free in order to be isomorphic to its complement.

Suppose that $|V(\Gamma)| \geq 6$. Then either $\Gamma$ or $\overline{\Gamma}$ has a triangle, corresponding to the simple Ramsey theory fact that $R(3, 3) = 6$. 

By examination we note that the only MPG on $5$ vertices is $C_5$, so thus $\Gamma = C_5$. 
\end{proof}

We use these two lemmas to conclude this section with its main theorem.

\begin{theoremN}
The only two self-complementary graphs in $\hat{\mathcal{G}}$ are $B_{2,2}$ and $C_5$.
\end{theoremN}

\begin{proof}
This follows immediately from Lemma \ref{2-complete sc} and Lemma \ref{MPG sc}
\end{proof}

\section{Enumeration of Reseminant Graphs}

We first define a family of minimal prime graphs first described in \cite{2015_REU_Paper}.

\begin{definition}
We say that a graph is \textbf{reseminant} if it can be generated from the 5-cycle via repeated vertex duplication. We let $\mathcal{R}$ denote the set of reseminant graphs.
\end{definition}

These reseminant graphs were originally considered in \cite{2015_REU_Paper}, where it was shown that reseminant graphs are minimal prime graphs. The goal of this section is to generalize and enumerate reseminant graphs.

Note that there exist minimal prime graphs which are not reseminant. One such example is on 8 vertices, illustrated below in Figure 6. To see it is not reseminant note that no vertex is a duplicate of another. To generalize the idea of reseminant graphs, we take the following definition:

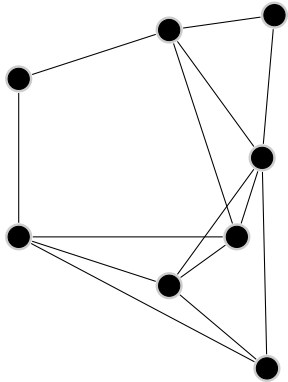
\begin{figure}[b] \label{non-reseminant graph}
    \centering
    \begin{tikzpicture}
    [scale=2, auto=left,every node/.style={circle,fill=black!20}]
    \node (n1) at (0, 1.376) {};
    \node (n2) at (1, 1.701) {};
    \node (n3) at (1.618, .8505) {};
    \node (n4) at (1, 0) {};
    \node (n5) at (0, 0.324919) {};
    \node (n6) at (1.65, -.55) {}; 
    \node (n7) at (1.45, 0.324919) {};
    \node (n8) at (1.7, 1.8) {};
    
    \draw [fill=black] (0,1.376) circle (2pt);
    \draw [fill=black] (1,1.701) circle (2pt);
    \draw [fill=black] (1.618,0.8505) circle (2pt);
    \draw [fill=black] (1,0) circle (2pt);
    \draw [fill=black] (0,0.324919) circle (2pt);
    \draw [fill=black] (1.45,0.324919) circle (2pt);
    \draw [fill=black] (1.65,-0.55) circle (2pt);
    \draw [fill=black] (1.7,1.8) circle (2pt);
    
    \foreach \from/\to in {n1/n2, n2/n3, n3/n4, n4/n5, n5/n1, n4/n6, n5/n6, n3/n6, n2/n7, n3/n7, n4/n7, n5/n7, n2/n8, n3/n8}
        \draw (\from) -- (\to);
\end{tikzpicture}
    \caption{A non-reseminant minimal prime graph, on 8 vertices}
    \label{fig:my_label}
\end{figure}

\begin{definition}
    A graph $\Gamma$ is a \textbf{base graph} if no two vertices are 'identical' - adjacent with the same adjacency relations to the other vertices. Given a vertex, we refer to the set of all vertices it is identical to, along with itself, as its cluster. 
\end{definition}
Another way of understanding base graphs is that its maximum cluster size is $1$. 
\begin{exampleN}
    The 5-cycle is a base graph. Its automorphism group is $D_{10}$. 
\end{exampleN}
\begin{definition}
    Given a base graph $\Gamma$, the set of $\Gamma$-reseminant graphs are those graphs which can be obtained from $\Gamma$ via a finite sequence of vertex duplications. 
\end{definition}
\begin{definition}
    Given a graph $\Gamma$, define its base to be the induced subgraph given by picking one vertex in each of its clusters. 
\end{definition}
\begin{propositionN}\label{base graphs and taking the base are nicely dual}
The following properties are true:
    \begin{enumerate}
        \item Taking the base of a graph is a well-defined operation on graphs
        \item Two isomorphic graphs have isomorphic base graphs
        \item Let $\Gamma$ be a base graph. Then the base of a $\Gamma$-reseminant graph is $\Gamma$.
    \end{enumerate}
\end{propositionN}
\begin{proof}
    To see that taking the base is well-defined, note that picking any vertex in a cluster results in the same graph. This holds since the adjacency relations of two vertices in the same cluster are the same, via being identical.
    
    Given two isomorphic graphs, since taking the base graph is an operation only dependent on the adjacency relations of the graph, they must have isomorphic base graphs. 
    
    Let $\Gamma$ be a base graph, and suppose we have a $\Gamma$-reseminant graph $\hat{\Gamma}$. Then for each vertex $v$ in $\Gamma$, $\hat{\Gamma}$ has a cluster for $v$ with size dependent on the number of vertex duplications done on $v$. It remains to prove that two clusters cannot be merged via the operation of vertex duplication. Let $v, v'$ be vertices in $\Gamma$. Then without loss of generality there exists a vertex $x$ that $v$ is connected to but $v'$ is not. Since vertex duplication duplicates adjacency relations, each vertex in the cluster of $v$ will be connected to each vertex in the cluster of $x$ in $\Gamma$, and no vertex in the cluster of $v'$ will be connected to a vertex in the cluster of $x$. This shows that the cluster of $v$ and $v'$ in $\hat{\Gamma}$ are distinct. 
\end{proof}
\begin{lemmaN} \label{graph isomorphisms determined by their base graph}
    Let $\Gamma$ and $\Lambda$ be base graphs.
    Let $\hat{\Gamma}$ be a $\Gamma$-reseminant graph, and identify $\Gamma$ as a subgraph of $\hat{\Gamma}$ via taking a base. Let $\hat{\Lambda}$ be a $\Lambda$-reseminant graph. Given an isomorphism $\phi: \hat{\Gamma} \to \hat{\Lambda}$, the image of $\Gamma$ under $\phi$ is a base of $\Lambda$. 
    
    Conversely given an isomorphism $\psi: \Gamma \to \Lambda$ where $\Lambda$ is a base of $\hat{\Lambda}$, $\psi$ extends to an isomorphism $\hat{\psi}: \hat{\Gamma} \to \hat{\Lambda}$ if for every vertex $v \in \Gamma$, the size of the cluster of $v$ is equal to the size of the cluster of $\psi(v)$ in $\hat{\Lambda}$ 
\end{lemmaN}
\begin{proof}
    Consider the image $\phi(\Gamma)$. Suppose $v_1, v_2$ in this image are in the same cluster. Then they are in the same cluster in $\Gamma$, a contradiction. Thus the image is a base graph. 
    
    Now suppose we have an isomorphism $\psi: \Gamma \to \Lambda$ as described, and that for every vertex $v \in \Gamma$, the size of the cluster of $v$ is equal to the cluster of $\psi(v)$ in $\hat{\Lambda}$. We can define the isomorphism $\hat{\psi}: \hat{\Gamma} \to \hat{\Lambda}$ which extends $\psi$ by taking, for each $v \in \hat{\Gamma}$ the cluster of $v$ to the cluster of $\psi(v)$ bijectively. This defines a graph map on the entirety of $\hat{\Gamma}$ since $\Gamma$ is the base of $\hat{\Gamma}$. It is easy to verify the adjacency relations are preserved, thus $\hat{\psi}$ is an isomorphism.  
\end{proof}
From here we have the tools to prove more general statements analogous to those proven about the 5-cycle as a base graph.
\begin{lemmaN} \label{G-reseminant graphs surjectively determined by sequences}
    Let $\Gamma$ be a base graph, with $n$ vertices, which are numbered $1, ..., n$. Then a finite sequence $R_m$ of numbers with range $\{1, ..., n\}$ defines a $\Gamma$-reseminant graph given by performing the duplications according to the sequence, and every $\Gamma$-reseminant graph can be defined in this way. 
\end{lemmaN}
\begin{proof}
    The process for going from a sequence $R_m$ to a $\Gamma$-reseminant graph is clear. To see this obtains all $\Gamma$-reseminant graphs, note that given a $\Gamma$-reseminant graph $\hat{\Gamma}$, we can identify its base $\Gamma$ inside the graph of $\Gamma'$ by taking the induced subgraph given by picking one vertex from each cluster, as in \ref{base graphs and taking the base are nicely dual}. Then any vertex $v'$ lives inside the cluster of some vertex $v$ of $\Gamma$ (identified as a subgraph of $\Gamma'$), and performing a vertex duplication on $v'$ is the same as performing a vertex duplication on $v$, since they lie in the same cluster.   
\end{proof}
\begin{lemmaN} \label{order doesn't matter for general reseminant graphs}
    Let $R_m$, $S_m$ be two sequences that define a $\Gamma$-reseminant graph. $R_m$ and $S_m$ define the same reseminant graph if some reordering of the sequence $R_m$ yields $S_m$.
\end{lemmaN}
\begin{proof}
    Take the identity isormophism on $\Gamma$. This extends to an isomorphism via lemma \ref{graph isomorphisms determined by their base graph} since their vertex cluster sizes are the same. 
\end{proof}
\begin{theoremN} \label{G-reseminant graphs enumerated by orbits of action of Aut(G) on n-tuples}
    Consider the automorphism graph of $\Gamma$, $Aut(\Gamma)$ as a subgroup of $S_n$, based on its action on the labeled vertices. Let $Aut(\Gamma)$ act on the set of n-tuples via the induced action from $S_n$. Then two n-tuples define isomorphic $\Gamma$-reseminant graphs if and only if they lie in the same orbit.
\end{theoremN}
\begin{proof}
    First suppose that two n-tuples $(a_1, ..., a_n)$ and $(b_1, ..., b_n)$ lie in the same orbit, ie there exists $\sigma \in Aut(\Gamma)$ such that $\sigma(a_1, ..., a_n) = (b_1, ..., b_n)$. Then $\sigma$ is an isomorphism on the base graph of each $\Gamma$-reseminant graph, and it extends to an isomorphism via lemma \ref{graph isomorphisms determined by their base graph} since lying in the same orbit tells us $\sigma$ satisfies the proper cluster size condition. 
    
    Now suppose that two n-tuples define isormophic $\Gamma$-reseminant graphs. Given an isomorphism $\phi$, which is an isomorphism when restricted to $\Gamma$, call it $\varphi$. The image $\phi(\Gamma)$ is necessarily also $\Gamma$. In addition, it holds that for every vertex $v$ in $\Gamma$, its cluster size is equal to the cluster size of $\phi(v)$. Thus taking $\varphi$ shows that the two $n$-tuples lie in the same orbit. 
\end{proof}
\begin{corollaryN} \label{enumearting G-reseminant graphs on n+k vertices}
    Let $Aut(\Gamma)$ act on the set of non-negative n-tuples with sum $k$. Then the number of $\Gamma$-reseminant graphs on $n+k$ vertices is equal to the number of orbits of that action. 
\end{corollaryN}
\begin{proof}
    This follows easily from the fact that the elements of $Aut(\Gamma)$ do not change the sum of the n-tuple, so the set of non-negative n-tuples can be restricted to those yielding graphs with $n+k$ vertices.
\end{proof}
    For a particular $\Gamma$, we can compute the exact number via a standard orbit counting argument (Cauchy-Frobenius formula, also often (falsely) referred to as
    Burnside's Lemma). We do this for the 5-cycle to count all reseminant graphs. 
\begin{corollaryN}\label{enumerating reseminant graphs}
    The number of reseminant graphs on $n+5$ vertices is given by
    \begin{itemize}
        \item $\cfrac{\binom{n+4}{n}+4+\frac{5n^2+30n+40}{8}}{10}$ if $2 \mid n, 5 \mid n$
        \item $\cfrac{\binom{n+4}{n}+\frac{5n^2+30n+40}{8}}{10}$ if $2\mid n, 5 \nmid n$ 
        \item $\cfrac{\binom{n+4}{n}+4+\frac{5n^2+20n+15}{8}}{10}$ if $2 \nmid n, 5 \mid n$
        \item $\cfrac{\binom{n+4}{n}+\frac{5n^2+20n+15}{8}}{10}$ if $2 \nmid n, 5 \nmid n$ 
    \end{itemize}
\end{corollaryN}
\begin{proof}
    Let $X$ denote the set of non-negative 5-tuples $(a, b, c, d, e)$ such that their sum is $n$. Then $D_{10}$ acts on $X$ via considering $D_{10}$ as a subset of $S_5$. We wish to count the orbits of this action.
    
    We use the Cauchy-Frobenius orbit counting formula which states that the number of orbits is the the average number of fixed points, ie that
    \begin{equation} 
        |X / D_{10}| = \frac{1}{|D_{10}|} \sum_{\sigma \in D_{10}} |X^{\sigma}|
    \end{equation}
    Note that $X$ has size $\binom{n+5-1}{n}$, computed using a standard combinatorial counting argument (:balls \& boxes"). 
    Consider the elements of $D_{10}$. We have 
    \begin{itemize}
        \item 1 identity element which fixes all elements of $X$
        \item 4 different rotations which fix an element if $n \equiv 0 \text{ mod }5$
        \item 5 reflections, which fix $\frac{n^2+6n+8}{8}$ elements if $n$ is even, $\frac{n^2+4n+3}{8}$ elements if $n$ is odd
    \end{itemize}
    To see this, note that the non-trivial rotations fix an element if and only if all 5 of its indices are equal. 
    
    For a given reflection, consider the value of the fixed index to be $i$. Then $n-i$ must be even, and we get a fixed point for each pair $x, y$ such that $x+y = \frac{n-i}{2}$. Here, $x,y$ represent the values of the other 2 indices on one side of the reflection, and together with $i$ determine the 5-tuple, as illustrated in the figure. The number of such pairs $x, y$ is simply $\frac{n-i}{2} + 1$, so together with some casework on the parity 
    \begin{align*}
        & n \text{ even } & \sum_{2|i, 0 \leq i \leq n} \frac{n-i}{2} + 1 = \frac{n^2+6n+8}{8} \\
        & n \text{ odd } & \frac{(n-1)^2 + 6(n-1) + 8}{8} = \frac{n^2+4n+3}{8}
    \end{align*}
    this yields the above expressions. Then plugging into (1) yields the corollary. 
\end{proof}
We mention here that Corollary \ref{enumearting G-reseminant graphs on n+k vertices} was also obtained in Summer 2018 during a six week research project by a research team of (at the time) three high school students - Joshua Kolenbrander, Elijah Stroud, Selina Wu - under the guidance of the fourth author as part of the Honors Summer Math Camp at Texas State University. The team came up with the alternative (but equivalent) formula
 \begin{center}
            $\left\lfloor \frac{n^4}{240}+\frac{n^3}{24}+\frac{5 n^2}{24}+\frac{n}{16} (-1)^n+\frac{25 n}{48}\right\rfloor +1$
        \end{center}
for the number of reseminant graphs on $n$ vertices. Their proof rested on the observation that the number of reseminant graphs is the number of bracelets on $n$ beads where exactly 5 beads (corresponding to the original 5-cycle) are black. It was also during this project that the
term "reseminant" was coined.


\section{Asymptotic Enumeration}
Though we enumerated the family of reseminant graphs in Section 7, enumerating minimal prime graphs in general is more difficult. We find an enumerative connection of maximal triangle-free, 3-colorable graphs, and utilize a folklore construction to attain a asymptotic lower bound. To do this, we consider a \textit{third} notion of minimality, which removed the condition of being connected.

\begin{definition}
A graph $\Gamma$ is a possibly disconnected minimal prime graph (PDMPG) if 
    \begin{enumerate}
        \item $\overline{\Gamma}$ is triangle-free
        \item $\overline{\Gamma}$ is 3-colorable
        \item removing any edge from $\Gamma$ results in a violation of (i) or (ii)
    \end{enumerate}
    We denote the class of such graphs as $\tilde{\mathcal{G}}$. 
\end{definition}

Essentially we are loosening the definition of minimal prime graphs to allow disconnected graphs. We immediately can characterize those graphs which lie in the set-difference, in other words the disconnected minimal prime graphs, as follows. 

\begin{lemmaN} \label{non-MPG PDMPG are disconnected with 2 complete components}
    Let $\Gamma$ be such that it is a possibly disconnected minimal prime graph but not a minimal prime graph. Then $\Gamma$ is disconnected with 2 completely connected components. 
\end{lemmaN}
\begin{proof}
    Let $\Gamma$ be such a graph. Then it is necessarily disconnected. If it has more than 2 disconnected components one can form a triangle in the complement. 
    
    Now consider these two disconnected components. It necessarily holds that each is complete, otherwise one can utilize the missing edge and a vertex from the other component to form a triangle in the complement. 
\end{proof}

Thus we can get a counting relation between the two types of graphs. Let $\mathcal{G}_k$ denote the set of minimal prime graphs on $k$ vertices, and $\tilde{\mathcal{G}}_k$ be defined similarly. 

\begin{corollaryN} \label{Counting the difference between MPG and PDMPG}
    $|\tilde{\mathcal{G}}_k \setminus \mathcal{G}_k| = \lfloor \frac{k-1}{2}\rfloor.$ 
\end{corollaryN}
\begin{proof}
    By Lemma \ref{non-MPG PDMPG are disconnected with 2 complete components}, $|\tilde{\mathcal{G}}_k \setminus \mathcal{G}_k|$ is equal to the number of graphs with 2 disconnected complete components. The corollary follows from casework on the parity of $k$ and counting pairs of positive numbers $(a, b)$ such that $a + b = k$ and $a \ne b$.  
\end{proof}

\begin{corollaryN} \label{Counting the difference between MPG and MCPG}
    Let $\hat{\mathcal{G}}_k$ denote the number of minimally connected prime graphs on $k$ vertices. Then if $k > 3$, $|\hat{\mathcal{G}}_k \setminus \mathcal{G}_k| = \lfloor \frac{k-1}{2}\rfloor - 1.$ When $k \leq 3$, $|\hat{\mathcal{G}}_k \setminus \mathcal{G}_k| = \lfloor \frac{k-1}{2}\rfloor.$
\end{corollaryN}
\begin{proof}
    Recall that by Theorem \ref{2-complete graphs are the set difference}, the set difference are the complete bridge graphs of the form $B_{m,n}$ for $m, n$ where $m \geq n > 1$ or $m \in \{1, 2\}, n = 1$. When $k \geq 3$ we can assume that $m \geq n > 1$. 
    
    Given a complete bridge graph $B_{m,n}$ with $m \geq n >1$ we can remove the bridge edge to obtain a disconnected graph with 2 complete connected components, with at least 2 vertices each. Conversely, given such a disconnected graph, we can add a bridge edge between the two components to obtain a complete bridge graph. This yields a set bijection to disconnected graphs with 2 complete components with at least 2 vertices each. Since there is only one exception in  $|\tilde{\mathcal{G}}_k \setminus \mathcal{G}_k|$, the graph with components of vertex size $1, k$, the corollary follows. The case when $k \leq 3$ can be easily checked by hand. 
\end{proof}

\begin{definition} A graph $\Gamma$ is a maximal triangle-free, 3-colorable graph if 
    \begin{enumerate} 
        \item $\Gamma$ is triangle-free
        \item $\Gamma$ is 3-colorable
        \item Adding any edge to $\Gamma$ results in a violation of (i) or (ii)
    \end{enumerate}
    We denote the class of such graphs by $\mathcal{H}$.
\end{definition}

\begin{propositionN} \label{Counting link to maximal triangle-free, 3-colorable graphs} 
    $\tilde{\mathcal{G}}$ is in bijection with $\mathcal{H}$ via taking the complement. This bijection preserves vertex number. 
\end{propositionN}
One way to interpret the above results is the enumerating minimal prime graphs is equivalent to enumerating maximal 3-colorable, triangle-free graphs. Given one, one can find the other. Before exploring the enumerative connection, we note that the bijection can be refined to minimal prime graphs.

\begin{lemmaN} The complement of a maximal triangle-free, 3-colorable graph $\Gamma$ with chromatic number 3 is a minimal prime graph. 
\end{lemmaN}
\begin{proof} We need to show that the complement of $\Gamma$, $\bar{\Gamma}$ is connected. Suppose not - then $\bar{\Gamma}$ is disconnected but still a possible disconnected minimal prime graph, so by \ref{non-MPG PDMPG are disconnected with 2 complete components} we see it is disconnected with 2 completely connected components. The complement of $\bar{\Gamma}$, returning to $\Gamma$, is thus a complete bipartite graph, which is 2-colorable. 
\end{proof}

Recall that it was proven that the complement of a minimal prime graph has chromatic number 3 in \cite[Lemma 2.3]{2015_REU_Paper}, which was used in Section 4. We give another proof of this statement.

\begin{lemmaN} The chromatic number of the complement of a minimal prime graph $\Gamma$ is $3$. 
\end{lemmaN}
\begin{proof}
We already have that the complement $\bar{\Gamma}$ is 3-colorable from the definition. Suppose $\bar{\Gamma}$ is 2-colorable. Then we can assign such a coloring, and note that the result is a subgraph of a complete bipartite graph. It cannot be the complete bipartite graph itself because then $\Gamma$ is disconnected. So it is a subgraph of the complete bipartite graph. But if there is more than 1 edge missing from being a complete bipartite graph, it can be added without violating being triangle-free or 2-colorable. Thus there is only one edge missing from being a complete bipartite graph, so the graph is of class $\mathcal{C}$ which we already showed are not minimal prime graphs. 
\end{proof}
\begin{theoremN} A graph $\Gamma$ is a minimal prime graph if and only if its complement is a maximal triangle-free 3-colorable graph with chromatic number 3. 
\end{theoremN}
\begin{proof} This follows directly from the two above Lemmas. 
\end{proof} 


\begin{remark}
    The set of maximal triangle-free, 3-colorable graphs is not either a superset or a subset of maximal triangle-free graphs. To see it is not a superset, take the Groetzsch graph, which is maximal triangle-free. To see it is not be a subset, note that a graph could be maximal triangle-free, 3-colorable but not maximal triangle-free, if it has an edge which, when added, makes it chromatic number 4 but does not add a triangle. Such a graph necessarily contains a pair of vertices $v_1, v_2$ with distance $\geq 3$ such that every 3-coloring of the graph results in $v_!, v_2$ having the same color. One such example, on 12 vertices, is given in \cite{Chris_counterexample}. 
\end{remark}

\begin{lemmaN}\label{asymptotic for labeled maximal triangle-free, 3-colorable graphs}
    The number of labeled maximal triangle-free, 3-colorable graphs on $n$ vertices is bounded below by $2^{(\frac{n^2}{8} + o(n^2))}$. 
\end{lemmaN}
\begin{proof}
    We adapt the "folklore" construction described in \cite{max_tri_free_construction}. A similar construction can be found in  \cite{asymptotic_max_tri-free}. 
    
    Let $n$ be given. Take any bipartite graph with parts of size $n/4, n/2$, call them $A$ and $B$ respectively. Then we embed it into a maximal triangle-free, 3-colorable graph, as follows. Take $A'$ a set of $n/4$ vertices, where for every $a \in A$ we introduce $a' \in A'$ which is adjacent to $a$, and every vertex $b \in B$ such that $a$ and $b$ are not adjacent. 
    
    This graph is triangle-free, and 3-colorable, if we let $A, A', B$ be the separate colors. However it might not be maximal with respect to those properties. 
    
    An edge cannot be added between any two vertices $a_1, a_2 \in A$, as they may have a mutual neighbor $b \in B$, which exists with proportion $1$ as $n \to \infty$. Similarly an edge cannot be introduced between two vertices $b_1, b_2 \in B$, due to a mutual neighbor $a \in A$. Given $a \in A, b\in B$, the edge $a, b$ cannot be added, as this forms a triangle on $(a, a', b)$. Given $a' \in A', b \in B$ which are not adjacent, $a', b$ cannot be added as this similarly forms a triangle on $(a, a', b)$. 
    
    So any edge added to this graph must occur between a vertex in $A$ and a vertex in $A'$, or between vertices in $A'$. However, we can simply add such edges until the graph is maximal, 3-colorable, to see that we have embedded our original bipartite graph in a maximal triangle-free, 3-colorable graph on $n$ vertices. Since there are $2^{\frac{n^2}{8}}$ such graphs, this yields the lower bound of $2^{\frac{n^2}{8} + o(n^2)}$ maximal triangle-free, 3-colorable graphs. 
\end{proof}

\section{Outlook}
The bulk of this paper establishes various properties of prime graphs of finite solvable groups, but much less is known about the prime graphs of general finite groups. A few facts have been established - the prime graphs of the simple groups are known \cite{Kondratev_Simple_Groups}, and a prime graph of a finite group cannot be either a chain of length greater than 4 \cite{chain_length_prime_groups}, or 5 disconnected points \cite{Kondratev_cocliques}. Generally speaking, very little is known about which graphs are realizable as prime graphs of finite groups, which reflects the difficulty of understanding finite groups in comparison with solvable groups. Nevertheless, properties of prime graphs of solvable groups give rise to interesting conjectures and questions about realizability of graphs as the prime graph of some finite group. One such example is Maslova's Conjecture \cite{khukhro2014unsolved}, which she Gorshkov proved for almost simple groups \cite{Maslova_almost_simple}. 
\begin{conjectureN}[Maslova's Conjecture]
    Given a finite group $G$ such that its prime graph does not contain 3-cocliques, its prime graph is isomorphic to the Gruenberg-Kegel graph of some finite solvable group.
\end{conjectureN}
From our work, one can ask similar questions about finite groups, for example
\begin{enumerate}
    \item Is there a finite group $G$ such that its prime graph has more than 5 vertices and is self-complimentary?
    \item Is there a finite group $G$ such that its prime graph is not Hamiltonian and its complement has chromatic number $\geq 3$?
\end{enumerate}

Another direction being explored, primarily by the second author, is the spectral theory of the minimal prime graphs of finite solvable groups. This is the content of a separate paper, titled \textit{The Adjacency Spectra of Some Families of Minimally Connected Prime Graphs}, being written at the time of writing.

We expect that the lower bound given by \ref{asymptotic for labeled maximal triangle-free, 3-colorable graphs} should be sharp, since it is exactly the asymptotic for maximal triangle-free graphs \cite{asymptotic_max_tri-free}. However, by the prior remark, the set of maximal triangle-free, 3-colorable graphs is neither a subset or a superset of maximal triangle-free graphs. Ideally, one might be able to asymtotically approximate the maximal triangle-free, 3-colorable graphs which are not maximal triangle-free by filling in the allowed edges towards a maximal triangle-free graph. Although this process might over-count, asymptotically we expect it should not be an issue. 

\section{Acknowledgements}

This research was conducted under NSF-REU grant DMS $1005206$ by the first, second, third, and fifth authors during the Summer of 2020 under the supervision of the fourth author.  The authors graciously acknowledge the financial support of NSF, as well as the hospitality of Texas State University.  In particular, Dr. Yong Yang, the director of the REU program, is thanked for persevering despite the conditions created by the pandemic and running the REU program in full.  

\newpage

\printbibliography

\end{document}